\documentclass{article}
\usepackage[utf8]{inputenc}
\usepackage{amsmath}
\usepackage{amsthm}
\usepackage{amsfonts}
\usepackage{mathtools}
\usepackage{wasysym}
\usepackage{MnSymbol}
\usepackage{thmtools}
\usepackage{stmaryrd}
\usepackage[letterpaper,margin=1in]{geometry}  
\usepackage{slashed}
\usepackage[english]{babel}		
\usepackage[hyphens]{url}
\usepackage[hyperindex,breaklinks]{hyperref}
\usepackage{bookmark}					
\usepackage[T1]{fontenc}
\usepackage{xspace}		
\usepackage{fancyhdr}
\usepackage{enumerate}
\usepackage{mathrsfs}
\usepackage{graphicx}
\usepackage{soul,color}
\usepackage{mathtools}
\usepackage{tikz-cd}
\usepackage[maxbibnames=99]{biblatex}
\usepackage{csquotes}
\usepackage{chngcntr}
\usepackage[bbgreekl]{mathbbol}
\counterwithin{equation}{section}
\usepackage{todonotes}
\DeclareSymbolFontAlphabet{\mathbb}{AMSb}
\DeclareSymbolFontAlphabet{\mathbbl}{bbold}

\renewcommand{\epsilon}{\varepsilon}
\renewcommand{\rho}{\varrho}
\renewcommand{\phi}{\varphi}
\newcommand{\NN}{\ensuremath{\mathbb{N}}\xspace}
\newcommand{\ZZ}{\ensuremath{\mathbb{Z}}\xspace}

\newcommand{\FF}{\ensuremath{\mathbb{F}}\xspace}
\newcommand{\TT}{\ensuremath{\mathbb{T}}\xspace}

\newcommand{\DD}{\ensuremath{\mathbbl{\Delta}}\xspace}

\newcommand{\tc}{\ensuremath{\mathrm{TC}}}
\newcommand{\thh}{\ensuremath{\mathrm{THH}}}
\newcommand{\tp}{\ensuremath{\mathrm{TP}}}
\newcommand{\tr}{\ensuremath{\mathrm{TR}}}

\DeclarePairedDelimiter\floor{\lfloor}{\rfloor}

\newtheorem{theorem}{Theorem}

\newtheorem{thm}{Theorem}[section]

\newtheorem{lem}[thm]{Lemma}
\newtheorem{cor}[thm]{Corollary}

\newtheorem{conj}[thm]{Conjecture}

\theoremstyle{remark}
\newtheorem{rem}[thm]{Remark}

\newtheorem{con}[thm]{Construction}

\title{\texorpdfstring{$\tr$ of quasiregular semiperfect rings is even}{TR of quasiregular semiperfect rings is even}}

\author{Micah Darrell, Noah Riggenbach}

\begin{document}
\maketitle
\begin{abstract}
We show that $\tr_{2i+1}(S)=0$ for all $i\in \NN$ and all $S$ quasiregular semiperfect. 
\end{abstract}
\tableofcontents

\section{Introduction}

In this paper we prove that the topological restriction homology $\tr(A)$ is locally even in the characteristic $p$ quasi-syntomic topology. In fact in the characteristic $p$ situation a stronger result is true.

\begin{theorem}[Theorem~\ref{thm: A in text}]~\label{thm: A}
    Let $S$ be a quasiregular semiperfect $\mathbb{F}_p$-algebra. Then $\tr_{2i-1}(S)=0$ for all $i\in \NN$. 
\end{theorem}

After finishing the work on this paper we found out that Devalapurkar and Mondal have independently proven the same result by different methods. 

This stronger statement is a characteristic $p$ phenomenon. The much weaker statement of $\mathrm{TR}(R)$ being even for perfectoid rings is known to be false in mixed characteristic, for example if $\mathcal{O}_{C}$ is the ring of integers in a complete and algebraically closed field $C/\mathbb{Q}_p$ then $\mathrm{TR}_1(\mathcal{O}_{C})=0$ only when $C$ is spherically complete. Nevertheless the fact that for perfectoid rings this evenness does appear locally and that several (but not all) of the arguments we make below have analogues in mixed characteristic motivate us to make the following conjecture.

\begin{conj}
    Let $A$ be a quasisyntomic ring. Then there exists some quasisyntomic cover $A\to S$ such that $\tr(S)$ is concentrated in even degrees.
\end{conj}

Here is an overview of our proof:

Hesselholt in \cite{Hesselholt_curves} and McCandless in \cite{McCandless_curves} prove that
\[
\mathrm{TR}(A;\ZZ_p) \simeq \varprojlim_{e} \Omega \mathrm{K}\left(A[t]/t^e, (t); \ZZ_p \right),
\]
where the right hand side is known as the curves on $\mathrm{K}$-theory of $A$. Here $\tr(A;\ZZ_p)$ is the $p$-completion of integral topological restriction homology and is related to $p$-typical topological restriction homology by a canonical idempotent splitting $\tr(A;\ZZ_p)\simeq \prod_{m\in I_p}\tr(A)$ where $I_p$ is the set of positive integers coprime to $p$. The Dundas-Goodwillie-McCarthy theorem \cite[Theorem 7.0.0.1]{Dundas_Goodwillie_McCarthy} provides an identification 
\[
\mathrm{K}_{*}(A[t]/t^e, (t)) \cong \mathrm{TC}_{*}(A[t]/t^e, (t)),
\]
which allows us to compute $\mathrm{TR}(A;\ZZ_p)$ as
\[
\mathrm{TR}(A;\ZZ_p) \simeq \varprojlim_{e} \Omega \mathrm{TC}\left(A[t]/t^e, (t);\ZZ_p\right).
\]

In \cite{BMS2} Bhatt, Morrow, and Scholze construct a filtration on $\mathrm{TC}$, known as the motivic filtration, for $i \in \mathbb{N}$, the $i$'th graded piece of this filtration on $\mathrm{TC}(A)$ is denoted $\mathbb{Z}_p(i)(A)[2i]$ and 
\[
\mathbb{Z}_p(i)(A) \cong \mathrm{fib}\left( \phi_i - \mathrm{can} : \mathcal{N}^{\geq i} \widehat{\DD}_A\{i\} \to \widehat{\DD}_A\{i\} \right). \\
\]
Here $\widehat{\DD}$ is the Nygaard-completed prismatic cohomology, $\mathcal{N}^{\geq i}\DD$ is the $i$'th piece of the Nygaard filtration, and $\{i\}$ denotes the $i$'th Breuil-Kisin twist, which we will ignore moving forward since we are working over $\mathbb{F}_p$ and these line bundles are canonically trivializable.

For an $\mathbb{F}_p$-algebra $A$, there is an identification $\widehat{\DD}_A \simeq \mathrm{LW}\Omega_A$, between the prismatic cohomology and the derived de Rham Witt complex. If the ring $A$ has a $\delta$-lift to characteristic 0, it is explained by Mathew in \cite{Mathew_recent_advances} how to compute the derived de Rham Witt complex $\mathrm{LW}\Omega_A$ and its Nygaard filtration, using the divided power de Rham complex of the lift. 

We write $S_e$ for the $\mathbb{F}_p$-algebra 
\[
S_e \cong k[y_{,}^{1/p^{\infty}} x]/(y, x^e),
\]
for a perfect $\mathbb{F}_p$-algebra $k$. This lets us identify 
\[
\mathrm{LW}\Omega_{S_e} \simeq \left( W(k)\left[y^{1/p^{\infty}}, \frac{y^n}{n!}\right]\left[x, \frac{x^{em}}{m!}\right] \xrightarrow{ \ \  d \ \ } W(k)\left[y^{1/p^{\infty}}, \frac{y^n}{n!}\right]\left[x, \frac{x^{em}}{m!}\right]\ dx \right)^{\wedge}_{p}
\]
where $m$ and $n$ range across $\mathbb{N}$  as in \cite{Mathew_recent_advances} and the work of Sulyma in \cite{Sulyma_truncated}. 

From here we can explicitly compute the cohomology of the complexes $\mathbb{Z}_p(i)(S_e)$ for all $i, e$. These complexes only have cohomology in degree 1, so the BMS spectral sequence, which is just the spectral sequence of the motivic filtration on $\mathrm{TC}$, collapses and identifies 
\[
\mathrm{TC}_{2i - 1}\left(S_e, (x)\right) \cong \mathrm{H}^1\left( \mathbb{Z}_p(i)(S_e) \right). \\
\]

This identification allows us to explicitly compute the transition maps on the homotopy groups of the diagram
\[
\mathrm{TR}(k[y^{1/p^\infty}]/y;\ZZ_p) \simeq \varprojlim_{e} \Omega \mathrm{TC}\left(S_e, (t)\right),
\]
and from here we check by hand that these transition maps satisfy the Mittag-Leffler condition. This implies that $\mathrm{TR}(k[y^{1/p^\infty}]/y)$ is concentrated in even degrees. 

For a given set $T$ we will call the quasiregular semiperfect rings \[S_T:=k[y_t^{1/p^{\infty}}|t\in T]/(y_t)\] prototype quasiregular semiperfect rings. The calculations described above extend to the prototype quasiregular semiperfect rings case 
\[
S_{e,T} = S_T[x]/(x^e)
\]
and we also verify that $\mathrm{TR}(S_T)$ is concentrated in even degrees.

Every characteristic $p$ quasisyntomic ring admits a cover by a quasiregular semiperfect ring, and every quasiregular semiperfect ring $S$ admits a map $S_T\to S$ for some prototype quasiregular semiperfectoid ring such that $L_{S_T/k}[-1]\to L_{S/k}[-1]$ is surjective. Using the filtration on $\mathrm{TR}$ from \cite{riggenbach2023ktheory}, Theorem~\ref{thm: A} follows. Along the way we also get the following results on algebraic $\mathrm{K}$-theory which might be of independent interest.

\begin{theorem}[Theorem~\ref{thm: B in text}]
    Let $S_T=k[y_t^{1/p^{\infty}}|t\in T]/(y_t)$ where $T$ is some set, and let $S_{e,T}=S_T[x]/(x^e)$. Then for all $i\in \NN$ \[\mathrm{K}_{2i-1}(S_{e,T}, (x))\cong \bigoplus_{\substack{m \in I_p,  \\ \alpha \in \mathbb{N}[1/p]^{\oplus T}}} W_h(k)
\]
where $I_p$ is the set of integers coprime to $p$, and $h = h(p, i, e, m, \alpha)$. The groups $\mathrm{K}_{2i}(A_{e,T},(x))\cong 0$.
\end{theorem}

As a result of the above we obtain two Corollaries. The first is that we can express the derived de Rham-Witt forms purely in terms of prismatic cohomology.

\begin{cor}
    For $R$ an $\FF_p$-algebra the filtration of $\tr(R)$ constructed in \cite{riggenbach2023ktheory} agrees with the left Kan extension of the Postnikov filtration from smooth $\mathbb{F}_p$ algebras. In particular there are equivalences \[\mathrm{LW}\Omega^i_{R}\simeq \mathrm{fib}\left(\prod_{n\in \NN} \mathcal{N}^{\geq i}\widehat{\DD}_R/p^n\mathcal{N}^{\geq i+1}\widehat{\DD}_R\xrightarrow{\phi_i-can}\prod_{n\in \mathbb{N}}\widehat{\DD}_R/p^n\right)[i]\] for all simplicial commutative $\mathbb{F}_p$-algebras $R$. 
\end{cor}

After working on this paper we found out that Bhatt has obtained a result of this form by studying a cover of $R^{\mathrm{syn}}$.

The second Corollary we obtain is a new proof of a special case of Antieau-Mathew-Morrow-Nikolaus' cohomological bound on syntomic cohomology.
\begin{cor}[Theorem 5.1 for $\FF_p$-algebras, \cite{antieau2021beilinson}]~\label{cor: AMMN for char p}
    Let $R$ be an animated $\FF_p$-algebra. Then the syntomic cohomology $\ZZ_p(i)(R)$ is in cohomological degrees at most $i+1$. 
\end{cor}

\begin{rem}
    Note that Corollary~\ref{cor: AMMN for char p} together with \cite[Theorem 5.2]{antieau2021beilinson} in fact imply \cite[Theorem 5.1]{antieau2021beilinson} without the $\FF_p$-algebra assumption. Alternatively one could use Corollary~\ref{cor: AMMN for char p} together with the $p$-adic continuity result for syntomic cohomology of \cite[Proposition 7.4.10]{Bhatt_Lurie} to reduce the mixed characteristic statement to computing a Tot spectral sequence.
\end{rem}

\textbf{Acknowledgements} The first author would like to thank Ben Antieau for helpful discussions. The second author would like to thank Benjamin Antieau, Akhil Mathew, and Shubhodip Mondal for helpful conversations related to this paper. We would also like to thank Benjamin Antieau for comments and suggestions on an earlier version of this paper.

\section{The topological cyclic homology of truncated polynomial algebras over the prototype quasiregular semiperfect rings}
The goal of this section is to compute the syntomic cohomology for rings of the form $S_{e,T} = k[y_t^{1/p^{\infty}}|t\in T, x]/(y, x^e)$, where $k$ is a perfect $\mathbb{F}_p$-algebra and $T$ is some set. We will begin by first computing this in the case $T=\{1\}$. The key ideas for this calculation and the more general computation are the same, the only difference is that the case of general $T$ is more notationally challenging.

This is the truncated polynomial algebra on $x$ over the quasiregular semiperfect ring $S = k[y^{1/p^{\infty}}]/(y)$. Prismatic cohomology is symmetric monoidal, so we can compute 
\[
\DD_{S_{e,\{1\}}} \simeq \DD_S \otimes_{\mathbb{Z}_p}^L \DD_{\mathbb{F}_p[x]/x^e}.
\]

Since $S$ is quasiregular semiperfect, there is an identification $\DD_S \simeq \mathrm{A}_{crys}(S) \cong \mathrm{W}(k)[y^{1/p^{\infty}}, \frac{y^n}{n!}]^\wedge_p$, which is a flat $\mathbb{Z}_p$-algebra. Then using the identification of $\DD_{\mathbb{F}_p[x]/x^e}$ with the $p$-completed divided power de Rham complex as in \cite{Mathew_recent_advances} and \cite{Sulyma_truncated} we get
\[
\DD_{S_{e,\{1\}}} \simeq \left( W(k)\left[y^{1/p^{\infty}}, \frac{y^n}{n!}\right]\left[x, \frac{x^{ej}}{j!}\right] \xrightarrow{ \ \  d \ \ } W(k)\left[y^{1/p^{\infty}}, \frac{y^n}{n!}\right]\left[x, \frac{x^{ej}}{j!}\right] \ dx \right)^{\wedge}_{p}.
\]
Going forward we will ignore this $p$-completion. This is justified by two facts. The first is that both the canonical map and the divided Frobenius maps manifestly exist before $p$-completion in this case. The second point is that we are ultimately interested in the syntomic cohomology of $S_{e,\{1\}}$, ie the equalizer of the divided Frobenius and canonical maps. Since (derived) $p$-completion commutes with limits we may wait to take the $p$-completion until after computing this equalizer, but as we will see the syntomic cohomology will be bounded $p$-power torsion. 

From now on we take $n \in \mathbb{N}[1/p]$, so we can write $W(k)\left[y^{1/p^{\infty}}, \frac{y^n}{n!}\right]$ as $W(k)\left[\frac{y^n}{\floor*n!}\right]$. We use the ideas and notation from \cite{Sulyma_truncated} to write the above complex as
\[
\bigoplus_{m, n} W(k) \left\langle\frac{y^n}{\floor*n!} \frac{x^m}{\floor*{m/e}!} \right\rangle \xrightarrow{\ \ d \ \ } \bigoplus_{m, n} W(k) \left\langle \frac{y^n}{\floor*n!} \frac{x^m}{\lfloor (m-1)/e\rfloor!}\mathrm{dlog}x\right\rangle,
\]
up to $p$-completion, where $n$ ranges through $\mathbb{N}[1/p]$ and $m \in \mathbb{N}$.

Notice that this complex has a bi-grading by $m, n$ with a generator in degree $(m, n)$ given by $\frac{y^n}{\lfloor n \rfloor !} \frac{x^m}{\lfloor m/e \rfloor ! }$, this is a $\mathbb{N}[1/p]\times\mathbb{N}$ grading. Also recall that $d$ is a $W(k)\left[\frac{y^n}{\floor*n!}\right]$-linear differential since it is the extension of scalars of a $\ZZ_p$-linear differential along the map $\ZZ_p\to \DD_{S_{\{1\}}}\cong W(k)\left[\frac{y^n}{\floor*n!}\right]$. In particular $d\left(\frac{y^n}{\lfloor n \rfloor !} \frac{x^m}{\lfloor m/e\rfloor}\right) = \frac{y^n}{\lfloor n \rfloor !}d\left(\frac{x^m}{\lfloor m/e\rfloor}\right)$. The logarithmic differential $\mathrm{dlog} x$ satisfies $x * \mathrm{dlog} x = dx$, so $x^{m-1}\mathrm{d}x = x^m \mathrm{dlog}x$ which diagonalizes the differential. The point of these notations is to cleanly write the differential in terms of the generators
\begin{equation*}
\begin{split}
d\left(\frac{y^n}{\lfloor n \rfloor !}\frac{x^m}{\lfloor m/e \rfloor !}\right) & = m \frac{y^n}{\lfloor n \rfloor !} \frac{x^{m - 1}}{\lfloor \frac{m}{e} \rfloor !} dx \\
& = m \frac{\lfloor \frac{m - 1}{e} \rfloor !}{\lfloor m/e \rfloor !} \frac{y^n}{\lfloor n \rfloor !}\frac{x^{m - 1}}{\lfloor\frac{ m - 1}{e}\rfloor !} dx \\
& = \{ m, e \} \frac{y^n}{\lfloor n \rfloor !}\frac{x^m}{\lfloor (m-1)/e\rfloor!} \mathrm{dlog}x,
\end{split}
\end{equation*}
where
\[
\{ m, e \} := m \frac{\lfloor (m-1)/e \rfloor!}{\lfloor m/e \rfloor !} = 
\begin{cases}
m & e \nmid m \\
e & e \ | \ m.
\end{cases}
\]

To understand the Nygaard filtration we can identify $\mathrm{LW}\Omega_{S_{e,\{1\}}} \simeq \mathrm{A}_{crys}(S_{\{1\}}) \otimes_{\mathbb{Z}_p} \mathrm{LW}\Omega_{\mathbb{F}_p[x]/x^e}$, and since the Nygaard filtration is symmetric monoidal $\mathcal{N}^{\geq *} \DD_{S_{e,\{1\}}}\simeq \mathcal{N}^{\geq *}\mathrm{A}_{crys}(S_{\{1\}}) \otimes_{p^*\mathbb{Z}_p} \mathcal{N}^{\geq *}\mathrm{LW}\Omega_{\mathbb{F}_p[x]/x^e}$. Alternatively, $\mathcal{N}^{\geq i}\DD_{S_{e,\{1\}}}$ is the subcomplex where the Frobenius map $\phi$ is divisible by $p^i$. The Frobenius sends the degree $(m, n)$ generator $\frac{y^n}{\lfloor n \rfloor !} \frac{x^m}{\lfloor m/e \rfloor ! }$ to $\frac{y^{pn}}{\lfloor n \rfloor !} \frac{x^{pm}}{\lfloor m/e \rfloor ! }$ writing this in terms of the degree $(pm, pn)$ generator we get $$\frac{\lfloor pn \rfloor !}{{\lfloor n \rfloor !}}\frac{\lfloor pm/e \rfloor ! }{\lfloor m/e \rfloor ! }\frac{y^{pn}}{\lfloor pn \rfloor !}\frac{x^{pm}}{\lfloor pm/e \rfloor ! }.$$  Then using the Legendre formula to compute the $p$-adic valuation of these coefficients we see that $\mathcal{N}^{\geq i}\DD_{S_{e,\{1\}}}$ is given in bidegree $(m,n)$ by 
\[
p^{i - \lfloor m/e \rfloor - \lfloor n \rfloor }W(k)\left\langle\frac{y^n}{\floor*n!} \frac{x^m}{\floor*{m/e}!} \right\rangle \xrightarrow{\ \ d \ \ } p^{i - \lceil m/e \rceil - \lfloor n \rfloor }W(k)\left\langle \frac{y^n}{\floor*n!} \frac{x^m}{\lfloor (m-1)/e\rfloor!}\mathrm{dlog}x\right\rangle,
\]
if $i\geq \lceil m/e\rceil + \lfloor n\rfloor$ and $(\mathcal{N}^{\geq i}\DD_{S_{e,\{1\}}})_{(m,n)}=(\DD_{S_{e,\{1\}}})_{(m,n)}$ otherwise. 

We now need to compute the cohomology of $\mathrm{LW}\Omega_{S_{e,\{1\}}}$ and $\mathcal{N}^{\geq i}\mathrm{LW}\Omega_{S_{e,\{1\}}}$. In both cases $\mathrm{H}^0 = 0$, except in degrees $m=0$ since the differential is injective. The bigrading also passes to cohomology, so we have 
\[
\mathrm{H}^1(\mathrm{LW}\Omega_{S_{e,\{1\}}})_{(m, n)} = W(k) / \{ m, e \} \left\langle \frac{y^n}{\floor*n!} \frac{x^m}{\lfloor(m-1)/e\rfloor!}\mathrm{dlog}x\right\rangle,
\]
\[
\mathrm{H}^1(\mathcal{N}^{\geq i} \mathrm{LW}\Omega_{S_{e,\{1\}}})_{(m, n)} = \begin{cases} W(k) / p^{\lceil m/e\rceil - \lfloor m/e\rfloor} \{ m, e \} \left\langle p^{i - \lceil m/e \rceil - \lfloor n \rfloor} \frac{y^n}{\floor*n!} \frac{x^m}{\lfloor(m-1)/e\rfloor! }\mathrm{dlog}x\right\rangle & \textrm{ if }i\geq \lceil m/e\rceil + \lfloor n\rfloor\\
W(k) /\{ m, e \} \left\langle \frac{y^n}{\floor*n!} \frac{x^m}{\lfloor(m-1)/e\rfloor! }\mathrm{dlog}x\right\rangle & \textrm{ otherwise.} \end{cases}\]


Our next goal is to compute $\mathbb{Z}_p(i)(S_{e,\{1\}}) = \mathrm{fib}\left(\phi/p^i - \mathrm{can} : \mathcal{N}^{\geq i}\DD_{S_{e,\{1\}}} \to \DD_{S_{e,\{1\}}}\right)$, so now we can use the long exact sequence in cohomology to determine  $\mathrm{H}^k(\mathbb{Z}_p(i)(S_{e, \{1\}}))$ by computing the maps on cohomology \[(\phi/p^i - \mathrm{can}) : \mathrm{H}^k(\mathcal{N}^{\geq i}\mathrm{LW}\Omega_{S_{e,\{1\}}}) \to \mathrm{H}^k(\mathrm{LW}\Omega_{S_{e,\{1\}}}).\]

The canonical map comes from the inclusion, so it sends a degree $(m, n)$ generator $$p^{i - \lceil m/e \rceil - \lfloor n \rfloor }\frac{y^n}{\floor*n!} \frac{x^m}{\lfloor (m-1)/e\rfloor !}\mathrm{dlog}x \in \mathrm{H}^1(\mathcal{N}^{\geq i} \mathrm{LW}\Omega_{S_{e,\{1\}}})_{(m, n)}$$ to $$p^{i - \lceil m/e \rceil - \lfloor n \rfloor }\frac{y^n}{\floor*n!} \frac{x^m}{\lfloor (m-1)/e\rfloor!}\mathrm{dlog}x \in \mathrm{H}^1(\mathrm{LW}\Omega_{S_{e,\{1\}}})_{(m, n)}$$ which is $p^{i-\lceil m/e\rceil -\lfloor n\rfloor}$ times the generator of $\mathrm{H}^1(\mathrm{LW}\Omega_{S_{e,\{1\}}})$ if $i\geq \lceil m/e\rceil +\lfloor n\rfloor$ and is the identity map otherwise. 

Similarly, the divided Frobenius map sends a degree $(m, n)$ generator $$p^{i - \lceil m/e \rceil - \lfloor n \rfloor}\frac{y^n}{\floor*n!} \frac{x^m}{\lfloor (m-1)/e\rfloor!}\mathrm{dlog}x \in \mathrm{H}^1(\mathcal{N}^{\geq i} \mathrm{LW}\Omega_{S_{e,\{1\}}})_{(m, n)}$$ to $$p^{i - \lceil m/e \rceil - \lfloor n \rfloor - i + 1} \ \frac{y^{pn}}{\floor*n!} \frac{x^{pm}}{\lfloor(m-1)/e\rfloor !}\mathrm{dlog}x \in \mathrm{H}^1(\mathrm{LW}\Omega_{S_{e,\{1\}}})_{(pm, pn)}$$
where the $+1$ in the exponent of $p$ comes from 
\begin{equation*}
\mathrm{dlog}(x^p) = \frac{px^{p-1}}{x^p}dx = p \frac{1}{x}dx = p \mathrm{dlog}x. \\
\end{equation*}

Now writing this element in terms of the degree $(pm, pn)$ generator we get
\[
p^{\max\{i - \lceil m/e \rceil - \lfloor n \rfloor, 0\} - i + 1}\frac{y^{pn}}{\floor*n!} \frac{x^{pm}}{\lfloor(m-1)/e\rfloor!}\mathrm{dlog}x  = p^t \frac{y^{pn}}{\lfloor pn \rfloor !} \frac{x^{pm}}{\lfloor (pm-1)/e\rfloor !}\mathrm{dlog}x, 
\]
where $t$ is still to be determined. 
The Legendre formula (see \cite[Page 5]{Sulyma_truncated}) and properties of a valuation tell us that 
\begin{equation*}
\begin{split}
v_p \left( \frac{\lfloor pn \rfloor !}{\lfloor n \rfloor ! }\right)  & = \lfloor n \rfloor + v_p \left(\lfloor n \rfloor ! \right) - v_p \left( \lfloor  n \rfloor ! \right) \\ 
& = \lfloor n \rfloor,
\end{split}
\end{equation*}
and  
\begin{equation*}
\begin{split}
v_p \left( \frac{\lfloor (pm-1)/e \rfloor !}{\rfloor (m-1)/e \rfloor! }\right)  & = \lfloor m/e \rfloor + v_p\left( \frac{ \{pm, e\} }{ p \{m, e \} } \right) \\ 
& = \lceil m/e \rceil - 1.
\end{split}
\end{equation*}
So $t = \max\{i - \lceil m/e \rceil - \lfloor n \rfloor,0\} - i + 1 + \lfloor n \rfloor + \lceil m/e \rceil -1.$ If $\lceil m/e \rceil + \lfloor n \rfloor > i$, then $\max\{i - \lceil m/e \rceil - \lfloor n \rfloor,0\} = 0$ and $t = \lceil m/e \rceil + \lfloor n \rfloor - i$. Otherwise all of the terms cancel and $s = 0$. In other words $t = \max\{\lceil m/e \rceil + \lfloor n \rfloor - i,0\},$ and 
\[
\phi/p^i \left( \mathrm{H}^1(\mathcal{N}^{\geq i} \mathrm{LW}\Omega_{S_{e,\{1\}}})_{(m, n)} \right) = p^{\max\{\lceil m/e \rceil + \lfloor n \rfloor - i,0\}}\mathrm{H}^1(\mathrm{LW}\Omega_{S_{e,\{1\}}})_{(pm, pn)}.
\] 


Now fix $n$, and an $m$ such that $p \nmid m$. Using the idea from \cite{Hesselholt-Madsen}, we consider the following map
\[
\bigoplus_{a \geq 0} \mathrm{H}^1 \left( \mathcal{N}^{\geq i} \mathrm{LW}\Omega_{S_{e,\{1\}}} \right)_{(p^am, p^an)} \xrightarrow{\phi/p^i - \mathrm{can}} \bigoplus_{a \geq 0} \mathrm{H}^1 \left(\mathrm{LW}\Omega_{S_{e,\{1\}}} \right)_{(p^am, p^an)}.
\]
This is not a graded map, but is still a filtered map. For low values of $a$ the divided Frobenius map is an isomorphism $(p^am, p^an) \mapsto (p^{a + 1}m, p^{a + 1} n)$, but for large enough values of $a$ the canonical map will become an isomorphism $(p^am, p^an) \mapsto (p^am, p^an)$. 
Let $s$ be the smallest natural number such that 
\[
\left\lceil \frac{p^s m}{e} \right\rceil + \left\lfloor p^s n \right\rfloor > i,
\]
which we denote $s(p, i, e, m, n)$, or sometimes for convenience just $s$ with the dependence on $p, i, e, m, n$ left implicit. This is then the smallest value of $a$ in the above when the divided Frobenius is no longer an isomorphism. We may then use this to compute the syntomic complexes.

\begin{thm} 
For $S_{e,\{1\}} = k[y^{1/p^{\infty}}, x]/(y, x^e)$, $k$ a perfect $\mathbb{F}_p$-algebra, there is an isomorphism
\[
\mathrm{H}^1(\mathbb{Z}_p(i)(S_{e,\{1\}})) \cong \bigoplus_{\substack{m \in I_p, \\  n\in \mathbb{N}[1/p]}} W(k)/\{p^{s}m, e\},
\]
where $I_p$ is the set of integers coprime to $p$, and $s = s(p, i, e, m, n)$ is as above. The higher cohomology vanishes. 
\end{thm}

\begin{proof}
    By the above we know that with the exception of the $m=0$ terms the complexes $\mathcal{N}^{\geq i}\mathrm{LW}\Omega^*_{S_{e,\{1\}}}$ are concentrated in cohomological degree $1$ for all $i$. In addition the $m=0$ terms are exactly the terms coming from the inclusion $k[y^{1/p^\infty}]/y\to S_{e,\{1\}}$ which is split. From \cite[Lemma 8.19]{BMS2} we know that $\ZZ_p(i)(k[y^{1/p^\infty}]/y)$ is concentrated in degree $0$ by \cite[Lemma 8.19(1)]{BMS2} and so to compute $\mathrm{H}^1(\ZZ_p(i)(S_{e,\{1\}}))$ it is enough to compute the kernel (which is $\mathrm{H}^1$) and the cokernel (which is $\mathrm{H}^2$) of the maps \[\frac{\phi}{p^i}-can: \mathrm{H}^1(\mathcal{N}^{\geq i}\mathrm{LW}\Omega_{S_{e,\{1\}}})\to \mathrm{H}^1(\mathrm{LW}\Omega_{S_{e,\{1\}}})\] for all $i$. 

    We may write $\mathrm{H}^1(\mathcal{N}^{\geq i}\mathrm{LW}\Omega_{S_{e,\{1\}}})=\bigoplus_{m\in I_p}\mathrm{H}^1(\mathcal{N}^{\geq i}\mathrm{LW}\Omega_{S_{e,\{1\}}};m)$ where \[\mathrm{H}^1(\mathcal{N}^{\geq i}\mathrm{LW}\Omega_{S_{e,\{1\}}};m)=\bigoplus_{\substack{a\geq 0 \\ n\in \NN[1/p]}}\mathrm{H}^1(\mathcal{N}^{\geq i}\mathrm{LW}\Omega_{S_{e,\{1\}}})_{(p^am,n)}\] and the advantage of doing this is, as explained above, both the divided Frobenius and the canonical map will respect this this direct sum decomposition. Note that we may rewrite this as \[\mathrm{H}^1(\mathcal{N}^{\geq i}\mathrm{LW}\Omega_{S_{e,\{1\}}};m)=\bigoplus_{n\in \NN[1/p]}\mathrm{H}^1(\mathcal{N}^{\geq i}\mathrm{LW}\Omega_{S_{e,\{1\}}};m,n)\] with \[\mathrm{H}^1(\mathcal{N}^{\geq i}\mathrm{LW}\Omega_{S_{e,\{1\}}};m,n)=\bigoplus_{a\geq 0}\mathrm{H}^1(\mathcal{N}^{\geq i}\mathrm{LW}\Omega_{S_{e,\{1\}}})_{(p^am,p^an)}\] and that the Frobenius and canonical maps will decompose along this decomposition as well.\footnote{Note that when $a=0$ by assumption $p\nmid m$. In particular $\mathrm{H}^1(\mathcal{N}^{\geq i}\mathrm{LW}\Omega_{S_{e,\{1\}}})_{m,n}$ will not be the codomain of any divided Frobenius maps, so even though $n$ is still divisible by $p$ the Frobenius does not hit these terms.} Thus we need only determine the kernel and cokernel of these maps for fixed $m\in I_p$ and $n\in \NN[1/p]$, and we have already proven the direct sum decomposition appearing in the statement of the Theorem.

    We therefore need only consider $\mathrm{H}^1(\mathcal{N}^{\geq i}\mathrm{LW}\Omega_{S_{e,\{1\}}};m,n)$ for one fixed pair $(m,n)\in I_p\times \NN[1/p]$ at a time. To this end note that we have a commutative diagram
    \[
    \begin{tikzcd}
    \bigoplus\limits_{a\geq s}\mathrm{H}^1(\mathcal{N}^{\geq i}\DD_{S_{e,\{1\}}})_{(p^am, p^an)} \arrow[d, "\frac{\phi}{p^i}-can"] \arrow[r] & \mathrm{H}^1(\mathcal{N}^{\geq i}\DD_{S_{e,\{1\}}};m,n) \arrow[d, "\frac{\phi}{p^i}-can"] \arrow[r] & \bigoplus\limits_{0\leq a\leq s-1} \mathrm{H}^1(\mathcal{N}^{\geq 1}\DD_{S_{e,\{1\}}})_{(p^am, p^an)}\arrow[d] \\
    \bigoplus\limits_{a\geq s}\mathrm{H}^1(\DD_{S_{e,\{1\}}})_{(p^am, p^an)}  \arrow[r] & \mathrm{H}^1(\DD_{S_{e,\{1\}}};m,n) \arrow[r] & \bigoplus\limits_{0\leq a\leq s-1} \mathrm{H}^1(\DD_{S_{e,\{1\}}})_{(p^am, p^an)}
    \end{tikzcd}
    \] where the horizontal sequences are short exact sequences. First note that when $a\geq s$ the map $can$ is in fact the identity. In addition by the above discussion we have that the divided Frobenius map has $p$-divisible image for $a\geq s$, so the map on the left is an equivalence after $p$-completion since it is the identity mod $p$.

    Thus we need only compute the kernel and cokernel of the map \[\bigoplus\limits_{0\leq a\leq s-1} \mathrm{H}^1(\mathcal{N}^{\geq i}\DD_{S_{e,\{1\}}})_{(p^am,p^an)}\to \bigoplus\limits_{0\leq a\leq s-1} \mathrm{H}^1(\DD_{S_{e,\{1\}}})_{(p^am,p^an)}\] and for $a\neq s-1$ we know from the above that the divided Frobenius map is an equivalence in these degrees. In addition in these degrees $can$ is $p$-divisible by the above discussion and the $a=0$ summand in the target vanishes, so mod $p$ this map is surjective. This is a finite sum of $p$-complete groups and therefore is $p$-complete, and so this map is surjective before going mod $p$. 

    To compute the kernel of this map, consider the projection \[\bigoplus\limits_{0\leq a\leq s-1} \mathrm{H}^1(\mathcal{N}^{\geq i}\DD_{S_{e,\{1\}}})_{(p^am,p^an)}\xrightarrow{p_{s-1}} \mathrm{H}^1(\mathcal{N}^{\geq i}\DD_{S_{e, \{1\}}})_{(p^{s-1}m, p^{s-1}n)}\] which will induce an isomorphism after precomposing with the inclusion of the kernel. Let $(x_{s-1},\ldots, x_0)\in \ker(\frac{\phi}{p^i}-can)\cap \ker(p_s)$. Then $x_{s-1}=0$, but then $\phi/p^i(x_{s-2})=\phi/p^i(x_{s-2})-can(x_{s-1})=0$ and $\phi/p^i$ is an isomorphism in these degrees so $x_{s-2}=0$. Inductively $x_a=0$ for all $a$ and so the map \[\ker(\frac{\phi}{p^i}-can)\to \mathrm{H}^1(\mathcal{N}^{\geq i}\DD_{S_{e, \{1\}}})_{(p^{s-1}m, p^{s-1}n)}\] is injective.

    In order to show that this map is surjective we must build a representative of the generator of \[\mathrm{H}^1(\mathcal{N}^{\geq i}\DD_{S_{e,\{1\}}})_{(p^{s-1}m, p^{s-1}n)}\cong W(k)/\{p^sm,e\}\] in $\bigoplus_{0\leq a\leq s-1}\mathrm{H}^{1}(\mathcal{N}^{\geq 1}\DD_{S_{e,\{1\}}})_{(p^am,p^an)}$ which lies in the kernel of $\frac{\phi}{p^i}-can$. This identification is crucial for the next part of this paper, so we will record this identification seperately as Construction~\ref{con: generator of the kernel}. This completes the proof.
\end{proof}

These same methods can be applied to the quasiregular semiperfect rings of the form 
\[
S_{e,T} = k[y_t^{1/p^{\infty}}, x|t\in T]/(y_t, x^e)
\]
with $T$ some indexing set. To simplify things going forward we will use the multi-index notation $\alpha = (n_t|t\in T)$ where we always assume that there are only finitely many entries in $\alpha$ which are nonzero. This lets us write 
\[
\frac{y^{\alpha}}{\lfloor \alpha \rfloor !} := \prod_{t\in T}\frac{y_t^{n_t}}{\lfloor n_t\rfloor!}
\]
which is well defined because of the assumption that $\alpha$ has only finitely many nonzero entries. 

With this new notation in hand we have 
\[
\mathrm{LW}\Omega_{S_{e,T}} \simeq \bigoplus_{m, \alpha} W(k)\left\langle \frac{y^{\alpha}}{\lfloor \alpha \rfloor !}\frac{x^m}{\lfloor m/e \rfloor !} \right\rangle \xrightarrow{\ \ d \ \ } \bigoplus_{m, \alpha} W(k)\left\langle \frac{y^{\alpha}}{\lfloor \alpha \rfloor !}\frac{x^m}{\lfloor (m-1)/e \rfloor !} \right\rangle,
\]
where $\alpha$ ranges over all $T$-indexed multi-indices of elements in $\mathbb{N}[1/p]$. 

Just as before we get identifications
\[
\mathrm{H}^1(\mathrm{LW}\Omega_{S_{e,T}})_{(m, \alpha)} = W(k) / \{ m, e \} \left\langle \frac{y^{\alpha}}{\lfloor \alpha \rfloor !} \frac{x^m}{\lfloor (m-1)/e\rfloor!}\mathrm{dlog}x\right\rangle
\]
and
\begin{equation}\label{eqn: explicit generators}
\mathrm{H}^1(\mathcal{N}^{\geq i} \mathrm{LW}\Omega_{S_{e,T}})_{(m, \alpha)} = \begin{cases}
W(k) / p^{\lceil m/e\rceil-\lfloor m/e\rfloor} \{ m, e \} \left\langle p^{i - \lceil m/e \rceil - \lfloor \alpha \rfloor_{\ell_1}} \frac{y^{\alpha}}{\floor* \alpha!} \frac{x^m}{\lfloor (m-1)/e\rfloor!}\mathrm{dlog}x\right\rangle & \textrm{ if } i\geq \lceil \frac{m}{e}\rceil+\lfloor \alpha \rfloor_{\ell_1}\\
W(k)/\{m,e\}\left\langle\frac{y^{\alpha}}{\lfloor \alpha \rfloor !}\frac{x^m}{\lfloor (m-1)/e\rfloor!}\mathrm{dlog}x \right\rangle & \textrm{otherwise}
\end{cases},
\end{equation}
where $\lfloor \alpha \rfloor_{\ell_1} := \sum_{t\in T}\lfloor n_t \rfloor$ is the $\ell_1$ norm of $\lfloor \alpha \rfloor$.

Let $s(p,i,e, m, \alpha)$ be the minimal value of $a$ such that the divided Frobenius map is no longer an isomorphism between degrees $(p^am, p^a \alpha) \mapsto (p^{a+1}m, p^{a+1}\alpha)$, ie the smallest value such that \[\lceil \frac{p^sm}{e}\rceil +\lfloor p^s\alpha \rfloor_{\ell_1} > i\] which is always bounded above by the value of $s(p,i,e,m,0)$.

\begin{thm}\label{thm: B in text}
For $S_{e,T} = k[ y_t^{1/p^{\infty}}, x|t\in T]/(y_t, x^e)$, $k$ a perfect $\mathbb{F}_p$-algebra, there are isomorphisms
\[\mathrm{H}^0(\ZZ_p(i)(S_{e,T}))\cong \DD_{k[y_t^{1/p^{\infty}}]/(y_t)},\]
\[
\mathrm{H}^1(\mathbb{Z}_p(i)(S_{e,T})) \cong \bigoplus_{\substack{m \in I_p,  \\ \alpha \in \mathbb{Z}[1/p]^{\oplus T}}} W(k)/\{p^{s}m, e\},
\]
where $I_p$ is the set of integers coprime to $p$, $s = s(p, i, e, m, \alpha)$, and all the other cohomology groups vanish.
\end{thm}
\begin{proof}
    The proof is the same as in the $T=\{1\}$ case. 
\end{proof}

\begin{con}\label{con: generator of the kernel}
We will now construct the generator of \[\ker(\frac{\phi}{p^i}-can)_{(m,\alpha)}\subseteq \bigoplus_{0\leq a\leq s-1}\mathrm{H}^1(\mathcal{N}^{\geq i}\DD_{S_{e,T}})_{(p^am,p^a\alpha)}.\] Let \[g_{s-1}\in \mathrm{H}^1(\mathcal{N}^{\geq i}\DD_{S_{e,T}})_{(p^{s-1}m,p^{s-1}\alpha)}\cong W(k)/\{p^sm,e\} \left\langle p^{i - \lceil m/e \rceil - \lfloor \alpha \rfloor_{\ell_1}} \frac{y^{\alpha}}{\floor* \alpha!} \frac{x^m}{\lfloor (m-1)/e\rfloor!}\mathrm{dlog}x\right\rangle\] be the generator \[g_{s-1}=p^{i - \lceil m/e \rceil - \lfloor \alpha \rfloor_{\ell_1}} \frac{y^{\alpha}}{\floor* \alpha!} \frac{x^m}{\lfloor (m-1)/e\rfloor!}\mathrm{dlog}x\] as an element of the above group. Then for all $a< s-1$ we have that $\phi/p^i: \mathrm{H}^1(\mathcal{N}^{\geq i}\DD_{S_{e,T}})_{(p^am,p^a\alpha)}\to \mathrm{H}^1(\DD_{S_{e,T}})_{(p^{a+1}m,p^{a+1}\alpha)}$ is an equivalence and so there is some element $h_{s-2}\in \mathrm{H}^1(\mathcal{N}^{\geq 1}\DD_{S_{e,T}})_{(p^{s-2}m, p^{s-2}\alpha)}$ such that \[\phi(h_{s-2})/p^i=-can(g_{s-1})=-p^{i-\lceil p^{s-1}m/e\rceil -\lfloor p^{s-1}\alpha\rfloor_{\ell_1}}\frac{y^{p^{s-1}\alpha}}{\lfloor p^{s-1}\alpha \rfloor!}\frac{x^{p^{s-1}m}}{\lfloor (p^{s-1}m-1)/e\rfloor!}\mathrm{dlog}x\] and so unwinding the formulas gives $h_{s-2}=-p^{i-\lceil p^{s-1}m/e\rceil -\lfloor p^{s-1}\alpha\rfloor_{\ell_1}}g_{s-2}$ where $g_{s-2}$ is the indicated generator of $\mathrm{H}^1(\mathcal{N}^{\geq 1}\DD_{S_{e,T}})_{(p^{s-2}m, p^{s-2}\alpha)}$. Inductively we then have a sequence $(g_{s-1}, h_{s-2},\ldots, h_0)$ such that $\phi(h_a)/p^i=-can(h_{a+1})$ and the formula \[h_a=p^{\sum_{j=0}^{s-1-a}(i-\lceil p^{s-1-j}m/e \rceil-\lfloor p^{s-1-j}\alpha)\rfloor_{\ell_1}}g_a\] up to a unit where $g_a$ is the canonical generator of $\mathrm{H}^1(\mathcal{N}^{\geq 1}\DD_{S_{e,T}})_{(p^{a}m, p^{a}\alpha)}$. We will absorb this unit into $g_a$ by picking a different generator, it will not make a difference going forward.

Note that $\mathrm{H}^1(\DD_{S_{e,T}})_{m,\alpha}\cong W(k)/\{m,e\}=0$ when $m\in I_p$, and so $can(h_0)=0$. It follows that \[(g_{s-1},\ldots, p^{\sum_{j=0}^{s-1}(i-\lceil p^jm/e\rceil -\lfloor p^j\alpha\rfloor_{\ell_1})}g_0)\in \ker(\frac{\phi}{p^i}-can)\subseteq \bigoplus\limits_{0\leq a\leq s-1}\mathrm{H}^1(\mathcal{N}^{\geq i}\DD_{S_{e,T}})_{(p^am,p^a\alpha)}\] and since the map $\ker(\frac{\phi}{p^i}-can)\to \mathrm{H}^1(\mathcal{N}^{\geq i}\DD_{S_{e,T}})_{(p^{s-1}m,p^{s-1}\alpha)}$ is injective and the above element is mapped to a generator the above element is then a generator for the kernel.
\end{con}

\section{The topological restriction homology of the prototype quasiregular semiperfect rings}
In this section we compute topological restriction homology for the prototype quasiregular semiperfect rings. We will do this in two steps. The first step, which we have largely done already, will be to compute the relative topological cyclic homology of truncated polynomial algebras. The second step will be to compute the limit of these computations and apply \cite[Theorem A]{Hesselholt_curves}.

Recall the notation \[S_T:= k[y_t^{1/p^\infty}|t\in T]/(y_t);\hspace{2em} S_{e,T}:=S_T[x]/x^e\] where $k$ is some perfect $\FF_p$-algebra, $T$ is some indexing set, and $e\geq 0$ is some natural number. Then by \cite[Theorem 1.12(5)]{BMS2} there are spectral sequences of the form \[E_2^{i,j}=\mathrm{H}^{i-j}(\ZZ_p(-j)(S_T))\implies \pi_{-i-j}(\tc(S_T;\ZZ_p))\] and \[E^{i,j}_2=\mathrm{H}^{i-j}(\ZZ_p(-j)(S_{e,T}))\implies \pi_{-i-j}(\tc(S_{e,T};\ZZ_p))\] which are functorial in their input. In particular since there is a retraction $S_T\to S_{e,T}\to S_T$ we have that the second spectral sequence above decomposes into a direct sum of the first spectral sequence and a spectral sequence of the form
\[
E_2^{i, j} = \mathrm{H}^{i - j}(\widetilde{\mathbb{Z}_p}(-j)(S_{e,T})) \implies \pi_{-i -j}\widetilde{\mathrm{TC}}(S_{e,T}; \mathbb{Z}_p)
\]
where $\widetilde{\ZZ_p}(-j)(S_{e,T}):= \operatorname{cofib}(\ZZ_p(-j)(S_{T})\to \ZZ_p(-j)(S_{e,T}))$ and similarly for $\widetilde{\tc}(S_{e,T};\ZZ_p)$.
Since $\widetilde{\mathbb{Z}_p}(i)(S_{e,T})$ is concentrated in degree $1$ by Theorem~\ref{thm: B in text}, this spectral sequence degenerates and we get
\[
\widetilde{\mathrm{TC}}_{2i - 1}(S_{e,T};\ZZ_p) \cong \mathrm{H}^1(\mathbb{Z}_p(i)(S_{e,T}))
\]
and the even homotopy groups vanish.

From the Dundas-Goodwillie-McCarthy theorem \cite{Dundas_Goodwillie_McCarthy} we have 
\[
K_{*}(S_{e,T}, (x);\ZZ_p) \cong \widetilde{\mathrm{TC}}_{*}(S_{e,T};\ZZ_p)
\] and in fact we do not need to $p$-adically complete the $K$-theory side of this equivalence since $K_*(S_{e,T},(x))$ is bounded $p$-power torsion. This follows from \cite[Theorem D]{Land_Tamme} by taking $N=p$\footnote{For $A$ a ring and $I\subseteq A$ a two sided ideal, both notations $K(A,I)$ and $K(A, A/I)$ are used in the literature to mean the fiber of the map $K(A)\to K(A/I)$. We use the first in this paper and \cite{Land_Tamme} use the latter, but they mean the same thing.}.

We now would like to use the description of $\mathrm{TR}$ due to Hesselholt \cite[Theorem A]{Hesselholt_curves} and generalized by McCandless \cite{McCandless_curves}
\[
\mathrm{TR}(S_{T};\ZZ_p) \simeq \varprojlim_{e} \Omega K(S_{e,T}, (x))
\]
to understand $\mathrm{TR}(S_{T};\ZZ_p)$. Note that we mean the $p$-completion of the integral topological restriction homology here, ie $\tr(A;\ZZ_p)=\lim_{n}\thh(R;\ZZ_p)^{C_n}$ where the maps in $\mathbb{N}$ are given by divisibility. This is in contrast to the $p$-typical topological restriction homology which is our goal, namely $\tr(A)=\lim_{n}\thh(A)^{C_{p^n}}$. These are connected by the natural equivalence \[\tr(A;\ZZ_p)\simeq \prod_{m\in I_p}\tr(A)\] and so we do not lose any information when studying one versus the other. In particular we will show that $\tr(S_{e,T};\ZZ_p)$ is even which will imply that $\tr(S_{e,T})$ is also even. 

In order to show that $\tr(S_{T})$ is even, since we have already shown that $K(S_{e,T},(x))$ is concentrated in odd degrees it is then enough to show that each of the pro-systems $\{K_{2i-1}(S_{e,T},(x);\ZZ_p)\}_{e\in \mathbb{N}}$ is Mittag-Leffler. To this end we first determine what the maps in these pro-systems are. Since $\mathbb{N}\setminus p\NN$ is cofinal in $\mathbb{N}$ we will only need to consider the case where $p\nmid e$.

\begin{lem}
    Let $i\geq 0$, $e,f\in I_p$, and $f\geq e$ all be integers. Let $T$ be any set. Then the maps \[tr_{fe}:K_{2i-1}(S_{f,T},(x))\cong \bigoplus_{\substack{m\in I_p \\ \alpha\in \NN[1/p]^{\oplus T}}}W(k)/p^{s(p,i,f,m,\alpha)}\to K_{2i-1}(S_{e,T},(x))\cong \bigoplus_{\substack{m\in I_p\\ \alpha \in \NN[1/p]^{\oplus T}}}W(k)/p^{s(p,i,e,m,\alpha)}\] are all graded, and in bidegree $(m,\alpha)$ these maps are given by multiplication by 
    \begin{equation}\label{eqn: explicit formula for trfe}
    \frac{\lfloor (p^{s(p,i,e,m,\alpha)-1}m-1)/e\rfloor!}{\lfloor (p^{s(p,i,e,m,\alpha)-1}m-1)/f \rfloor!}p^{\lceil p^{s(p,i,e,m,\alpha)-1}m/e \rceil - \lceil p^{s(p,i,e,m,\alpha)-1}m/f \rceil  +\sum\limits_{j=s(p,i,e,m,\alpha)-1}^{s(p,i,f,m,\alpha)-1}(i-\lceil p^jm/f\rceil -\lfloor p^j\alpha \rfloor_{\ell_1})}
\end{equation}
\end{lem}

\begin{proof}
    We already have that the map $K_{2i-1}(S_{f,T},(x))\to \tc^{-}_{2i-1}(S_{f,T},(x))$ is injective for all $f$, so it is enough to determine a formula for the maps $\tc^{-}_{2i-1}(S_{f,T},(x))\to \tc^{-}_{2i-1}(S_{e,T},(x))$. This map is then a graded map and in bidegree $(p^km,\alpha)$ the map is given by multiplication by \[
\frac{\lfloor p^km-1 /e\rfloor!}{\lfloor p^km-1/f \rfloor!}p^{\max\{i - \lceil p^km/f \rceil - \lfloor p^k n \rfloor,0\} - \max\{i - \lceil p^km/e \rceil - \lfloor p^k n \rfloor,0\}}
\] as can be seen by using the explicit generators constructed in Equation~\ref{eqn: explicit generators}. The result then follows by the formula for the generators of $K_{2i-1}(S_{f,T},(x))$ and $K(S_{e,T},(x))$ constructed in Construction~\ref{con: generator of the kernel}.
\end{proof}

\subsection{\texorpdfstring{Evenness of $\tr(k[y_t^{1/p^\infty}|t\in T]/(y_t))$}{Evenness of prototypes}}

Now we would like to compute $\tr_{2i-1}(k[y_t^{1/p^\infty}|t\in T]/(y_t))$ by using the identifications 
\[
\tr_{2i-1}(k[y_t^{1/p^\infty}|t\in T]/(y_t)) \simeq \varprojlim_e \mathrm{H}^1\mathbb{Z}_p(i)\left(k[y_{t\:,}^{1/p^\infty} x|t\in T]/(y_t, x))\right).
\]

This case is slightly more difficult than that of $\mathrm{TR}(\mathbb{F}_p)$, because it is no longer the case that $s(e) \to \infty$ as $e \to \infty$, meaning the transition maps are non-trivial in some degrees.

\begin{lem}
For all $i\in \NN$, the limit diagram 
\[
\varprojlim_e \mathrm{H}^1\mathbb{Z}_p(i)\left(k[y_{t\:,}^{1/p^\infty} x|t\in T]/(y_t, x^e))\right),
\]
is Mittag-Leffler.
\end{lem}

\begin{proof}
Fix an $e \in \mathbb{N}$, we need to show that for some large enough value of $f,$ we have $\mathrm{im}(tr_{fe}) = \mathrm{im}(tr_{f'e})$ for all $f' \geq f$, i.e. the image of the transition maps stabilizes. For fixed $e$, there are only finitely many values of $m$ such that $\mathrm{H}^1\mathbb{Z}_p(i)(k[y_t^{1/{p^{\infty}}}, x|t\in T]/(y_t, x^e))_{(m, \alpha)} \neq 0$ for some $\alpha$. This lets us reduce to the case where there is only a single value of $m$. 

Note first that since $s(p,i,e,m,\alpha)$ is monotone increasing in $e$ we have that $s(p,i,f,m,\alpha)\geq s(p,i,e,m,\alpha)$. In particular the image of the maps $tr_{fe}$ is completely determined by how $p$-divisible the number in Equation~\ref{eqn: explicit formula for trfe} is. In particular it is enough to show that this number stabilizes for large enough $f$ uniformly in $\alpha$. 

In order to do this, note first that the function $s(p,i,f,m,\alpha)$ is monotone decreasing in each entry of $\alpha$, so taking $\alpha=(0,0,\ldots)$ we get a uniform upper bound $s(p,i,e,m,\alpha)\leq s(p,i,e,m)$. Let $f''$ be large enough so that $\lfloor\frac{p^{s(p,i,e,m)}m-1}{f''}\rfloor =0$. Then for any $f'>f''$ the only expression in Equation~\ref{eqn: explicit formula for trfe} that can change is \[\sum_{j=s(p,i,e,m,\alpha)-1}^{s(p,i,f',m,\alpha)-1}(i-\lceil p^jm/f'\rceil -\lfloor p^j\alpha\rfloor_{\ell_1})\] and we need only show that this sum does not change the image for $f'$ large enough uniformly in $\alpha$.

For this choose $f$ to be large enough that \[\lceil p^{2s(p,i,e,m)}m/f\rceil=1\] and that $f\geq f''$ as defined above. We claim that this is the uniform bound. To see this let $f'>f$ and let $\alpha\in \NN[1/p]^{\oplus T}$. We will whow that the image of $tr_{f'e}$ is the same as the image of $tr_{fe}$ by working one value of $\alpha$ at a time, which is sufficient since we have already chosen $f$ independent of $\alpha$. Fix an $\alpha\in \NN[1/p]^{\oplus T}$. Then there are two cases we need to consider: either $s(p,i,f,m,\alpha)\geq 2s(p,i,e,m)$ or it is not. 

Suppose that $s(p,i,f,m,\alpha)\geq 2s(p,i,e,m)$. Then we have that \begin{align*}\sum_{j=s(p,i,e,m,\alpha)-1}^{s(p,i,f',m,\alpha)-1}(i-\lceil p^jm/f'\rceil -\lfloor p^j\alpha\rfloor_{\ell_1}) &\geq \sum_{j=s(p,i,e,m,\alpha)-1}^{s(p,i,f',m,\alpha)-1}1\\
&= (s(p,i,f',m,\alpha)-1)-(s(p,i,e,m,\alpha)-1)\\
&\geq 2s(p,i,e,m) - s(p,i,e,m)\\
&=s(p,i,e,m)
\end{align*}
where the first inequality comes from the definition of the $s$ function and the second inequality comes from our assumption on $s(p,i,f,m,\alpha)$, the fact that  $s(p,i,f,m,\alpha)\leq s(p,i,f',m,\alpha)$, and the fact that $s$ is monotonically decreasing in entries in $\alpha$. In particular the image of $tr_{f'e}$ on this summand are at least $p^{s(p,i,e,m)}$-divisible and are in a $p^{s(p,i,e,m,\alpha)}$-torsion group, so on these summands this is the zero map which will be stable under increasing $f'$.

Now suppose that $s(p,i,f,m,\alpha)<2s(p,i,e,m)$. It then follows that for any such $\alpha$ and any $a<2s(p,i,e,m)$, \[
\left\lceil \frac{p^a m}{f'} \right\rceil + \lfloor p^a\alpha \rfloor_{\ell_1} = 1 + \lfloor p^a \alpha \rfloor_{\ell_1}
\]
since the first term will be smaller that $p^{2s(p,i,e,m)}m/f'<p^{2s(p,i,e,m)}m/f<1$. In particular the value of this sum does not depend on $f'$. In particular $a=s(p,i,f,m,\alpha)<2s(p,i,e,m)$ is the first time this sum is larger than $i$ and the sum if independent of $f$, so  $s(p,i,f',m,\alpha)=s(p,i,f,m,\alpha)$. This was the only way the sum $\sum_{j=s(p,i,e,m,\alpha)-1}^{s(p,i,f',m,\alpha)-1}(i-\lceil p^jm/f'\rceil -\lfloor p^j\alpha\rfloor_{\ell_1})$ depends on the value of $f'$ since by assumption $ p^jm/f'<p^{s(p,i,e,m)m}/f'<1$, so $im(tr_{fe})=im(tr_{f'e})$ on these summands as well. This completes the proof.

\end{proof}

So in particular

\begin{thm}\label{thm: main thm prototype}
    Let $k$ be a perfect $\mathbb{F}_p$-algebra. Then \[\tr_{2i-1}(k[y_t^{1/p^\infty}|t\in T]/(y_t))=0\] for all sets $T$ and all $i\in \ZZ$.
\end{thm}

\begin{proof}
Using the above identifications we find that 

\[
\tr_{2i-1}(k[y_t^{1/p^\infty}|t\in T]/(y_t)) \simeq \varprojlim_e \mathrm{H}^1\mathbb{Z}_p(i)\left(k[y_{t\:,}^{1/p^\infty} x|t\in T]/(y_t, x))\right).
\]
Only $\varprojlim^1$ terms can contribute to odd degrees in $\mathrm{TR}$, but this diagram is Mittag-Leffler so there can be no $\lim^1$ terms.
\end{proof}

\section{The topological restriction homology of quasiregular semiperfect rings}
We begin this section by recording a Lemma which we will need repeatedly. This Lemma is not due to us, and was proven as part of the proof of \cite[Proposition 8.12]{BMS2} and is recorded in \cite[Lemma 14.5]{Bhatt_Scholze}. We include it here for the convenience of the reader. 

\begin{lem}[\cite{Bhatt_Scholze}, Lemma 14.5]\label{lem: surj on cotangent implies surj on prism}
    Let $R$ be a perfectoid ring and $S$ and  $S'$ quasiregular semiperfectoid $R$-algebras. Let $S\to S'$ be a map which induces a surjection on $\pi_1(L_{S/R})\to \pi_1(L_{S'/R})$. Then the induced map \[\widehat{\DD}_S\to \widehat{\DD}_{S'}\] is also a surjection.\footnote{The references \cite[Proposition 8.12]{BMS2} and \cite[Lemma 14.5]{Bhatt_Scholze} only consider the map on prismatic cohomology before Nygaard completion. By a similar argument as in \cite[Proposition 7.4.6]{Bhatt_Lurie} we can ignore the Nygaard completion in the arguments that are to follow. Nevertheless the statement is also true for the Nygaard complete primsatic cohomology and so we will also record that proof as well. We would like to thank Ben Antieau for pointing out the proof of the Nygaard complete statement to us.}
\end{lem}
\begin{proof}
    By \cite[Lemma 4.25]{BMS2} $L_{S/R}$ and $L_{S'/R}$ have $p$-complete Tor-amplitude concentrated in (homological) degree $1$. Thus both $L_{S/R}[-1]$ and $L_{S'/R}[-1]$ are $p$-completely flat $S$ and $S'$ modules, respectively. It follows that after $p$-completion $L\Gamma^n_S(L_{S/R}[-1])=\Gamma^n_S(L_{S/R}[-1])$ and $L\Gamma^n_{S'}(L_{S'/R}[-1])=\Gamma_{S'}^n(L_{S'/R}[-1])$ where $\Gamma^n$ denotes the $n^{th}$ divided powers. Then by \cite[Proposition II.4.3.2.1(ii)]{Illusie} there are equivalences \[\bigwedge\nolimits_{S'}^n(L_{S'/R})\simeq L\Gamma^n_{S'}(L_{S'/R}[-1])[n]=\Gamma^n_{S'}(L_{S'/R}[-1])[n]\] and \[\bigwedge\nolimits_S^n(L_{S/R})\simeq L\Gamma^n_S(L_{S/R}[-1])[n]=\Gamma^n_S(L_{S/R}[-1])[n]\] for all $n\in \NN$.
    Since the map $L_{S/R}\to L_{S'/R}$ is surjective it then follows that all the maps $\bigwedge\nolimits^n_S(L_{S/R})[-n]\to \bigwedge\nolimits^n_S(L_{S/R})[-n]$ are surjective maps of discrete modules for all $n\in \NN$.

    Left deriving the Hodge-Tate comparison Theorem of \cite[Theorem 6.3]{Bhatt_Scholze} there are then filtrations on the Hodge-Tate cohomology $\mathrm{Fil}^{conj}_*\overline{\DD}_{-/\mathrm{A}_{inf}(R)}$ with $\mathrm{gr}^i\overline{\DD}_{-/\mathrm{A}_{inf}(R)}\simeq \bigwedge\nolimits^i_{-}L_{-/R}[-i]\{-i\}$. Thus the above shows that the map $\overline{\DD}_{S}\to \overline{\DD}_{S'}$ is surjective. Since both $\DD_S$ and $\DD_{S'}$ are $I$-adically complete, where $I$ is the Hodge-Tate divisor, it follows that the map $\DD_S\to \DD_{S'}$ is also surjective. 

    For the statement on the Nygaard complete prismatic cohomology, note that the argument above also shows that the maps $\mathcal{N}^i\DD_S\to \mathcal{N}^{i}\DD_{S'}$ are also surjective since they are finite pieces of the conjugate filtration. Thus the map $\thh(S;\ZZ_p)\to \thh(S';\ZZ_p)$ is surjective on homotopy groups and therefore the fiber of this map $F$ has homotopy concentrated in even degrees. The Tate $\TT$ spectral sequence for $F$ therefore collapses and we find that $F^{t\TT}$ also has homotopy concentrated in even degrees. From the fiber sequence \[F^{t\TT}\to \tp(S;\ZZ_p)\to \tp(S';\ZZ_p)\] we see that the failure of the map $\widehat{\DD}_S\cong \tp_0(S;\ZZ_p)\to \tp_0(S';\ZZ_p)\cong \widehat{\DD}_{S'}$ being surjective is measured by $\pi_{-1}(F^{t\TT})$ which we have just shown vanishes.
\end{proof}

We will use the above lemma to go from the prototype quaisregular semiperfect rings to the general case. The key step in this is then the next Lemma.

\begin{lem}\label{lem: surj on cotangent implies surj on odd tr}
    Let $R$ be a perfectoid ring and $S$ and $S'$ be quasiregular semiperfectoid $R$-algebras. Suppose there is a map $S\to S'$ which induces a surjection on $\pi_1(L_{S/R})\to \pi_1(L_{S'/R})$. Then for all $i\geq 1$ the induced map \[\tr_{2i-1}(S)\to \tr_{2i-1}(S')\] is surjective. 
\end{lem}
\begin{proof}
    From Lemma~\ref{lem: surj on cotangent implies surj on prism} we have that the map $\widehat{\DD}_S\to \widehat{\DD}_{S'}$ is a surjection. Let $d\in \mathrm{A}_{inf}(R)$ be an orientation. Then by \cite[Lemma 3.5]{Bhatt_Scholze} the Hodge-Tate divisor $I$ is $\Tilde{d}\DD_S$ in $\DD_S$ and $\Tilde{d}\DD_{S'}$ in $\DD_{S'}$. Consequently $I_n=I\cdot \phi^*(I)\cdot\ldots\cdot (\phi^{n-1})^*(I)$ is the ideal generated by $\Tilde{d}_n$ in both $\DD_S$ and $\DD_{S'}$. It then follows that for all $i,n\geq 0$ that the maps \[\widehat{\DD}_S\{i\}/I_n\to \widehat{\DD}_{S'}\{i\}/I_n\] are surjective since these maps are the map obtained by applying $-\otimes_{\mathrm{A}_{inf}(R)} \mathrm{A}_{inf}(R)/\Tilde{d}_n \{i\}$ to a surjective map and tensoring is right exact. 

    Recall from \cite[Corollary 7.2]{riggenbach2023ktheory} that there is a functorial filtration on $\tr(-)$ with associated graded given by \[\mathrm{gr}^i\tr(-):=\operatorname{Eq}\left(\begin{tikzcd}
        \prod_{n\geq 0}\mathcal{N}^{\geq i}\mathcal{O}_{\widehat{\DD}}/(\mathcal{N}^{\geq i+1}\mathcal{O}_{\widehat{\DD}}\otimes \mathcal{I}_n)\{i\} \arrow[r, shift left] \arrow[r, shift right] & \prod_{n\geq 0} \mathcal{O}_{\widehat{\DD}}/\mathcal{I}_n\{i\}
    \end{tikzcd}\right)[2i].\] This filtration is constructed by using the identification \[\tr(-)\simeq \operatorname{Eq}\left(\begin{tikzcd}
        \prod_{n\geq 0}\thh(-;\ZZ_p)^{hC_{p^n}} \arrow[r, shift left] \arrow[r, shift right] & \prod_{n\geq 0} \thh(-;\ZZ_p)^{tC_{p^n}}
    \end{tikzcd}\right)\] given by \cite[Remark 2.4.5]{McCandless_curves}. The terms appearing in this equalizer diagram are filtered and the filtration on $\tr(-)$ is then given by taking the equalizer in filtered spectra. For quasiregular semiperfectoid rings $\mathrm{gr}^i\tr(-)$ is concentrated in homological degrees $[2i-1,2i]$. Thus the spectral sequence associated to this filtration collapses and the equilizer and coequilizer of the above maps computes $\tr_{2i}(-)$ and $\tr_{2i-1}(-)$, respectively. 

    By functoriality there is then a commutative diagram \[\begin{tikzcd}
        \prod_{n\geq 0}(\mathcal{N}^{\geq i}\widehat{\DD}_S/(I_n\mathcal{N}^{\geq i+1}\widehat{\DD}_S))\{i\}\arrow[d] \arrow[r, shift left] \arrow[r, shift right] & \prod_{n\geq 0} \widehat{\DD}_{S}/I_n\{i\}\arrow[d]\\
        \prod_{n\geq 0}(\mathcal{N}^{\geq i}\widehat{\DD}_{S'}/(I_n\mathcal{N}^{\geq i+1}\widehat{\DD}_{S'}))\{i\} \arrow[r, shift left] \arrow[r, shift right] & \prod_{n\geq 0} \widehat{\DD}_{S'}/I_n\{i\}
    \end{tikzcd}\] and by the above the right vertical arrow is a surjection. Thus by the snake lemma the map on coequilizers is also surjective, and so the result follows.
\end{proof}

We are now ready to prove the main theorem.

\begin{thm}~\label{thm: A in text}
    Let $S$ be a quasiregular semiperfect $\mathbb{F}_p$-algebra. Then \[\tr_{2i-1}(S)=0\] for all $i$.
\end{thm}
\begin{proof}
    We will follow the reduction used in the proof of \cite[Proposition 8.12]{BMS2}. Let $S^\flat$ be the inverse limit perfection of $S$. Then $S^\flat\to S$ is surjective since $S$ is semiperfect. Let $J$ denote the kernel of $S^{\flat}\to S$. Then there is a map $S^{\flat}[x_j^{1/p^\infty}|j\in J]/(x_j)\to S$ which induces a surjection $\pi_1\left(L_{S^{\flat}[x_j^{1/p^\infty}|j\in J]/(x_j)/S^{\flat}}\right)\to \pi_1(L_{S/S^{\flat}})$. 
    
    Thus by Lemma~\ref{lem: surj on cotangent implies surj on odd tr} the maps \[\tr_{2i-1}(S^{\flat}[x_j^{1/p^\infty}|j\in J]/(x_j))\to \tr_{2i-1}(S)\] are surjective for all $i\in \ZZ$. The result then follows by applying Theorem~\ref{thm: main thm prototype}. 
\end{proof}

We are now also ready to compare this filtration to the filtration by de Rham-Witt forms. We record the construction of this filtration here but stress that these ideas are not due to us.

\begin{con}
    Consider the functor $\tr(-):\mathrm{Sch}_{\mathbb{F}_p}\to \mathrm{TCart}_p^\wedge$ where the image is the category of $V$-complete topological Cartier modules in the sense of \cite[Theorem 9]{Antieau_Nikolaus}. As such a functor $\tr(-)$ is left Kan extended from $\mathrm{Sm}_{\FF_p}$. To see this note that by completeness it is enough to check modulo $V$, but then $\tr(-)/V\simeq \thh(-)$ and so is left Kan extended as desired.

    Consider now the functor $F^{\geq *}_H\tr(-):\mathrm{Sm}_{\FF_p}\to \operatorname{DF}(\mathrm{TCart}_p^\wedge)$ given by $F^{\geq *}\tr(-)=\tau_{\geq *}\tr(-)$ and by \cite[Theorem C]{Hesselholt_curves} we can identify $\mathrm{gr}^*_{H}\tr(-)\simeq W\Omega^i_{-}[i]$ with the usual $V$-adic completion. Left Kan extending then gives an exhaustive filtration $F^{\geq *}_{H}\tr(-)$ for all $\FF_p$-schemes such that $\mathrm{gr}^*\tr(-)\simeq (\mathrm{LW}\Omega^i_{-})^\wedge[i]$ where the completion is the derived $V$-adic completion. 
\end{con}

\begin{cor}~\label{cor: hesselholt filtration=BMS filtration}
    Let $R$ be an animated $\mathbb{F}_p$-algebra. Then the filtration of $\mathrm{TR}(R)$ given in \cite[Corollary 7.2]{riggenbach2023ktheory} is equivalent to the filtration $F_H^{\geq *}\tr(R):\mathrm{CAlg}_{\FF_p}\to \widehat{\operatorname{DF}}(\mathrm{TCart_p}^\wedge)$ described above. 
\end{cor}

\begin{proof}
    Denote by $F^{\geq *}_{BMS}\tr(R)$ the filtration constructed in \cite[Corollary 7.2]{riggenbach2023ktheory} and $F^{\geq *}_{H}\tr(R)$ the filtration outlined in the statement of the Corollary. Then $F^{*}_{BMS}\tr(R)$ is a quasisyntomic sheaf by definition and $F^{\geq *}_{H}\tr(R)$ is a quasisyntomic sheaf since $F^{\geq -\infty}
    _H\tr(R)\simeq \tr(R)$ and $\mathrm{gr}^i_{H}\tr(R)=(\mathrm{LW}\Omega^i_R)^\wedge_V[i]$ are quasisyntomic sheaves. Thus to show these are equivalent it is enough to prove this for quasiregular semiperfect $\FF_p$-algebras.

    By definition for $R$ quasiregular semiperfect we have that $L_{R/\FF_p}$ is a flat $R$-module in homological degree $1$. Consequently $\bigwedge^iL_{R/\FF_p}$ is concentrated in homological degree $i$ by the same argument as in Lemma~\ref{lem: surj on cotangent implies surj on prism}. Inductively we have that $\mathrm{LW}_n\Omega_R^i$ is concentrated in homological degree $i$, and therefore $\mathrm{LW}\Omega_R^i$ is concentrated in homological degrees $[i+1,i]$. In particular the spectral sequence associated to this filtration has no room for differentials, and by Theorem~\ref{thm: A} we have that $\mathrm{LW}\Omega_R^i$ is discrete and agrees with $\tr_{2i}(R)$. Both this filtration and the filtration $F^{\geq *}_{BMS}\tr(R)$ are then the double speed Postnikov filtration and are therefore naturally equivalent. The result follows.
\end{proof}

Using this result we can now also prove Corollary~\ref{cor: AMMN for char p}.

\begin{cor}
    Let $R$ be an animated $\FF_p$-algebra. Then the syntomic cohomology $\ZZ_p(i)(R)$ is concentrated in cohomological degrees at most $i+1$.
\end{cor}
\begin{proof}
    From Corollary~\ref{cor: hesselholt filtration=BMS filtration} we have that \[\mathrm{LW}\Omega^i_R[i]\simeq \mathrm{gr}^i_H \tr(R)\simeq \mathrm{gr}^i_{BMS}\tr(R)\simeq \mathrm{fib}\left(\prod_{n\in \NN} \mathcal{N}^{\geq i}\DD_R/p^n\mathcal{N}^{\geq i+1}\DD_R\xrightarrow{\phi_i-can}\prod_{n\in \mathbb{N}}\DD_R/p^n\right)[2i]\] and so $\mathrm{fib}\left(\prod_{n\in \NN} \mathcal{N}^{\geq i}\DD_R/p^n\mathcal{N}^{\geq i+1}\DD_R\xrightarrow{\phi_i-can}\prod_{n\in \mathbb{N}}\DD_R/p^n\right)$ is concentrated in cohomological degrees at most $i$. The syntomic cohomology $\ZZ_p(i)(R)$ is then given by the fiber of the map \[1-F:\mathrm{fib}\left(\prod_{n\in \NN} \mathcal{N}^{\geq i}\DD_R/p^n\mathcal{N}^{\geq i+1}\DD_R\xrightarrow{\phi_i-can}\prod_{n\in \mathbb{N}}\DD_R/p^n\right)\to \mathrm{fib}\left(\prod_{n\in \NN} \mathcal{N}^{\geq i}\DD_R/p^n\mathcal{N}^{\geq i+1}\DD_R\xrightarrow{\phi_i-can}\prod_{n\in \mathbb{N}}\DD_R/p^n\right)\] where $F$ is the map induced by the mod $p$ maps $\mathcal{N}^{\geq i}\DD_R/p^n\mathcal{N}^{\geq i+1}\DD_R\to \mathcal{N}^{\geq i}\DD_R/p^{n-1}\mathcal{N}^{\geq i+1}\DD_R$ and $\DD_R/p^n\to \DD_R/p^{n-1}$. The fiber can introduce a nonzero cohomology group in degree $i+1$ but no higher, so the result follows.
\end{proof}

\printbibliography
\end{document}